\documentclass{amsart}

\usepackage{amssymb,amsfonts, latexsym,amsmath,hhline,array,longtable}
\usepackage{pdfsync,color, comment,colortbl, mathrsfs,stmaryrd,cite,graphicx}

\usepackage{ifpdf}
\ifpdf \usepackage[colorlinks=true, citecolor=blue, linkcolor=blue, urlcolor=blue]{hyperref} \fi

\newtheorem{formula}{}[section]
\newtheorem{definition}[formula]{Definition}
\newtheorem{corollary}[formula]{Corollary}
\newtheorem{remark}[formula]{Remark}
\newtheorem{lemma}[formula]{Lemma}
\newtheorem{theorem}[formula]{Theorem}
\newtheorem{prop}[formula]{Proposition}

\newtheorem{construction}{Construction}

\def\thrm{\begin{theorem}}
\def\thrml#1{\begin{theorem}\label{#1}}
\def\ethrm{\end{theorem}}
\def\rmrk{\begin{remark}}
\def\rmrkl#1{\begin{remark}\label{#1}}
\def\ermrk{\end{remark}}
\def\dfntn{\begin{definition}}
\def\dfntnl#1{\begin{definition}\label{#1}}
\def\edfntn{\end{definition}}
\def\nmrt{\begin{enumerate}}
\def\enmrt{\end{enumerate}}

\def\qtnl#1{\begin{equation}\label{#1}}
\def\eqtn{\end{equation}}
\def\lmm{\begin{lemma}}
\def\lmml#1{\begin{lemma}\label{#1}}
\def\elmm{\end{lemma}}
\def\crllr{\begin{corollary}}
\def\crllrl#1{\begin{corollary}\label{#1}}
\def\ecrllr{\end{corollary}}
\def\css{\begin{cases}}
\def\ecss{\end{cases}}
\def\prf{\begin{proof}}
\def\eprf{\end{proof}}

\begin{document}

\title[SRGs decomposable into a DDG and a Hoffman coclique]{Strongly regular graphs decomposable into a~divisible design graph and a Hoffman coclique}
\author{}
\address{}
\author{Alexander L. Gavrilyuk}
\address{Shimane University, Matsue, Japan}
\email{gavrilyuk@riko.shimane-u.ac.jp}
\author{Vladislav V. Kabanov}
\address{Krasovskii Institute of Mathematics and Mechanics, Yekaterinburg, Russia}
\email{vvk@imm.uran.ru}
\thanks{The research of Alexander Gavrilyuk is supported by JSPS KAKENHI Grant Number 22K03403.}
\date{}

\begin{abstract}
In 2022, the second author found a prolific construction 
of strongly regular graphs, which is based on joining a coclique and  
a divisible design graph with certain parameters.
The construction produces strongly regular graphs 
with the same parameters as the complement of the symplectic 
graph $\mathsf{Sp}(2d,q)$.

In this paper, we determine the parameters of strongly regular 
graphs which admit a decomposition into a divisible design graph 
and a coclique attaining the Hoffman bound. 
In particular, it is shown that when the least eigenvalue 
of such a strongly regular graph is a prime power, its parameters 
coincide with those of the complement of $\mathsf{Sp}(2d,q)$. 
Furthermore, a generalization of the construction is discussed. 
\end{abstract}

\maketitle

\section{Introduction}\label{sec1}

In \cite{VKa}, the second author of the present paper proposed 
a construction of strongly regular graphs. It goes as follows: 
take a divisible design graph, say $\Delta$, with specific parameters, 
add an independent set of vertices (a coclique) $C$ whose cardinality 
depends on the parameters of $\Delta$, 
and join their vertices in a certain way governed by a symmetric 2-design 
on $|C|$ points. The resulting graph, say $\Gamma$, has the parameters 
of the complement of the symplectic graph $\mathsf{Sp}(2d,q)$.
In an earlier paper \cite{VK}, a ({\em prolific}) construction 
of hyperexponentially many non-isomorphic divisible design graphs 
with the parameters of $\Delta$ was discovered; thus, combining these 
two results gives a prolific construction of strongly regular graphs 
with the parameters of $\Gamma$. 

A natural question arises: can one use divisible design graphs 
whose parameters are different from those required in \cite{VKa} 
(in order to obtain strongly regular graphs whose parameters 
are different from those of the complement of~$\mathsf{Sp}(2d,q)$)?
Note that the coclique $C$ is a Hoffman coclique 
in $\Gamma$, i.e., its cardinality satisfies the Delsarte-Hoffman bound. 
In this paper, we analyse the situation when a strongly regular graph 
$\Gamma$ contains a Hoffman coclique $C$ such that the subgraph 
induced on $\Gamma\setminus C$ is a divisible design graph. 
Our main theorem essentially answers the above question as follows.

\begin{theorem}\label{theo:main}
Let $\Gamma$ be a primitive strongly regular graph with parameters $(v,k,\lambda,\mu)$ 
and the least eigenvalue $s$. Suppose that $\Gamma$ contains a coclique $C$ 
of cardinality $\frac{vs}{s-k}$ and the subgraph 
induced on $\Gamma\setminus C$ is a proper divisible design graph. Then 
there exists a natural number $n$ such that the parameters of 
$\Gamma$ satisfy:
\[
v=\frac{(-s)(n^2-1)}{n+s},\quad 
k=(-s)n,\quad 
\lambda=\mu=(-s)(n+s).
\]
\end{theorem}

\begin{corollary}\label{coro}
In the notation of Theorem \ref{theo:main}, if $-s$ is a prime power, 
then $n$ also is a prime power divisible by $-s$. 
In particular, with $q=n/(-s)$, 
$s=-q^d$, $n=q^{d+1}$ for some natural $d$, 
and $\Gamma$ has the parameters of the complement of~$\mathsf{Sp}(2d,q)$.
\end{corollary}

We will recall the necessary definitions and notation 
in Section \ref{sect:2}. 
The proof of Theorem \ref{theo:main} is based on the following observation. 
On the one hand, the spectrum of the graph induced by $\Gamma\setminus C$ 
is completely determined in terms of the spectrum of $\Gamma$ 
and contains at most four distinct eigenvalues. 
On the other hand, the spectrum of a divisible design graph $\Delta$
contains at most five distinct eigenvalues; however, their multiplicities 
in general are not determined by the parameters of $\Delta$.  
Comparing these spectra and using various feasibility conditions 
yields the result; its proof occupies Section \ref{sect:3}.

In Section \ref{sect:4}, we describe a construction of a graph $\Gamma$ 
satisfying the hypothesis of Theorem \ref{theo:main}. When $-s$ is not 
a prime power, it slightly generalizes the one given in \cite{VKa}; 
however, no examples are known to us: we will discuss the smallest open case
which could be constructed in this way. 

Note that the decomposition of $\Gamma$ as in Theorem \ref{theo:main} 
is a special case of the so-called regular decomposition (equitable partition) of strongly regular graphs~\cite{H}, see also \cite[Section~1.1.13]{BW}. 
Strongly regular graphs with strongly regular decomposition 
were studied by Haemers and Higman in~\cite{HH}:   
in particular, they characterized such graphs when 
one of the two parts is a coclique and the other one 
is also a strongly regular graph. 

Note also that a Hoffman coclique in $\Gamma$ becomes a Delsarte clique 
in the complement $\overline{\Gamma}$ of $\Gamma$, whose cardinality attains 
the Delsarte clique bound. However, our result does not apply 
to a decomposition of a strongly regular graph into a clique and 
a divisible design graph (see \cite[Section~5]{VKa}), 
since the complement of the latter one is not a divisible design graph 
in general. This type of a decomposition will be studied 
elsewhere.

We conclude the introduction with the following remark to put this work 
into a larger context. In 1971, Wallis \cite{WW} 
proposed a construction of strongly regular graphs based on 
affine designs and a Steiner 2-design: it takes all points 
of some affine designs as the vertices and adds edges between them 
in a way governed by the Steiner design. 
Thirty years later Fon-Der-Flaass \cite{FF} independently reinvented 
this idea: he proposed three constructions and showed that they produce 
hyperexponentially many strongly regular graphs. 
The first Fon-Der-Flaass construction is a special case of the Wallis' one 
when a Steiner design is complete, i.e., has blocks of size 2. 
In particular, when one takes affine planes of order $q$ 
as the affine designs, the resulting graphs have the parameters of   
a generalized quadrangle $\mathsf{GQ}(q+1,q-1)$.
The second and third constructions explicitly use affine planes
and produce graphs with the parameters of $\mathsf{GQ}(q,q)$ and 
$\mathsf{GQ}(q-1,q+1)$, respectively. 
Muzychuk \cite{MM} went further and generalized the idea of Wallis 
by replacing a Steiner 2-design with a partial linear space 
whose collinearity graph is itself strongly regular. 
In this way, he built at least 6 families of strongly regular graphs 
including and generalizing the first and third Fon-Der-Flaass constructions: 
but not the second one. This missing generalization seems to 
be given by the construction from \cite{VKa}. 

We also remark that quite a few other constructions 
of strongly regular graphs with the parameters of $\mathsf{Sp}(2d,q)$ 
have been found: Abiad and Haemers \cite{AH} and Kubota \cite{Ku} 
applied Godsil-McKay switching when $q=2$, Ihringer \cite{Ih} 
and Brouwer, Ihringer, and Kantor \cite{BIK} developed a sort of 
switching in polar spaces.

\section{Strongly regular and divisible design graphs}\label{sect:2}

All basic definitions and results on strongly regular graphs 
can be found in~\cite{BW}. 
Let $\Gamma$ be a primitive strongly 
regular graph with parameters $(v,k,\lambda,\mu)$ and 
spectrum $k^1,r^f,s^g$, where $k>r>s$ and the exponents 
mean the multiplicities. 
Recall that a coclique (an independent set) $C$ in $\Gamma$ has cardinality 
not exceeding $vs/(s-k)$, due to Delsarte \cite[Section 3.3]{D} and 
Hoffman \cite{AJH}. 
In the special case when 
$C$ is a {\bf Hoffman coclique}, i.e., 
$|C|$ attains this upper bound, 
one can determine the spectrum of the subgraph induced on $\Gamma\setminus C$ (hereinafter for a subset $X$ of the vertex set of $\Gamma$ 
we also write $X$ for the subgraph of $\Gamma$ induced by $X$). 

\begin{prop}\label{vd}
{\rm (\cite[Theorem 2.4]{HH}, \cite[]{vD})}  
Let $\Gamma$ be a primitive strongly regular graph on $v$ vertices, with spectrum $k^1, r^f, s^g$, 
where $k>r>s$. If $\Gamma$ contains a Hoffman coclique $C$ (of size $c=vs/(s-k)$), 
then the induced subgraph $\Gamma\setminus C$ is a regular, connected graph with spectrum
$$(k+s)^1,r^{f-c+1},(r+s)^{c-1},s^{g-c},$$
and it has four distinct eigenvalues if $c<g$. 
\end{prop}

Divisible design graphs were introduced in \cite{HKM} (see also \cite{CH}).  
A $K$-regular graph on $V$ vertices, which is not complete or edgeless, 
is a {\bf divisible design graph} with parameters $(V,K,\lambda_1 ,\lambda_2; m,n)$ 
if its vertex set can be partitioned into $m$ classes of size $n$, 
in such a way that any two different vertices from the same class have  $\lambda_1$ 
common neighbours, and any two vertices from different classes 
have  $\lambda_2$ common neighbours. 
A divisible design graph is {\bf proper} unless 
$m = 1$, $n = 1$, or $\lambda_1 = \lambda_2$, in which cases the graph 
is strongly regular. 

The partition of a divisible design graph into  classes is called  a {\bf canonical partition}. Let $\Delta$ be a divisible design graph with parameters $(V,K,\lambda_1,\lambda_2; m,n)$. 
By~\cite[Lemma 2.1]{HKM}, the spectrum of $\Delta$ is 
\begin{equation}\label{eq:ddgspec}
K^1,\sqrt{K-\lambda_1}^{f_1},-\sqrt{K-\lambda_1}^{f_2},
\sqrt{K^2 -\lambda_2 V}^{g_1},-\sqrt{K^2 -\lambda_2 V}^{g_2},    
\end{equation}
where 
\begin{equation}\label{eq:ddgmult}
f_1 + f_2 = m(n-1),\quad g_1 + g_2 = m-1.   
\end{equation}
Note, however, that some of these eigenvalues may coincide or 
be missing if their multiplicities are equal to $0$.

\section{Proof of Theorem \ref{theo:main}}\label{sect:3}

In this section we suppose that $\Gamma$ satisfies the hypothesis of 
Theorem \ref{theo:main} and let $\Delta$ denote a divisible design graph 
with parameters $(V,K,\lambda_1,\lambda_2;m,n)$,  
induced on $\Gamma\setminus C$. 
Put $c=\frac{vs}{s-k}$. 
Proposition~\ref{vd} and 
Eq. \eqref{eq:ddgspec} then imply that the following two multisets: 
$$\{(k + s)^1, r^{f-c+1}, (r + s)^{c-1}, s^{g-c}\}$$
and 
$$\{K^1,\sqrt{K-\lambda_1}^{f_1},-\sqrt{K-\lambda_1}^{f_2},\sqrt{K^2 -\lambda_2 V}^{g_1},-\sqrt{K^2 -\lambda_2 V}^{g_2}\}$$
coincide. 
We will consider this situation in a series of lemmas below.
Recall (see \cite[Theorem~9.1.2]{BH}) that a regular graph (not complete or edgeless) 
is strongly regular if and only if its spectrum has precisely three distinct eigenvalues. 
Since we assume that $\Delta$ is a proper 
divisible design graph, $c<g$ holds, i.e., the spectrum of 
$\Gamma\setminus C$ contains four distinct eigenvalues. 
It then follows that in Eq. \eqref{eq:ddgspec} 
one of the multiplicities ($f_1,f_2,g_1$, or $g_2$) vanishes 
or some two of the four eigenvalues (distinct from $K$) coincide.
In the former case, the possible spectra of $\Delta$,  
corresponding to $r^{f-c+1}$, $(r + s)^{c-1}$, $s^{g-c}$, 
are given in Table \ref{tb:8cases}, where we order the eigenvalues in the descending order
(as $r>r+s>s$). In the latter case, since $\Delta$ has 
four distinct eigenvalues, 
\begin{center}
either 
$\sqrt{K-\lambda_1}=-\sqrt{K-\lambda_1}=0$ or $\sqrt{K^2 -\lambda_2 V}=-\sqrt{K^2 -\lambda_2 V}=0$    
\end{center}
holds, which implies that $r+s=0$. Note that this equality 
also holds in Cases (2), (3), (6), (7) from Table \ref{tb:8cases}. 
In Lemmas \ref{lm:case1}--\ref{lm:case67} below, we first consider 
the cases with $r+s\ne 0$ from Table \ref{tb:8cases}. 

\begin{table}[h]
\caption{The possible spectra of $\Delta$ with one of the multiplicities vanished.}\label{tb:8cases}
\begin{tabular}{ccccc}
(1) & $\sqrt{K-\lambda_1}^{f_1}$    & $\sqrt{K^2-\lambda_2V}^{g_1}$ & $-\sqrt{K^2-\lambda_2V}^{g_2}$ & $-\sqrt{K-\lambda_1}^{0}$    \\
(2) & $\sqrt{K-\lambda_1}^{f_1}$    & $\sqrt{K^2-\lambda_2V}^{g_1}$ & $-\sqrt{K^2-\lambda_2V}^{0}$ & $-\sqrt{K-\lambda_1}^{f_2}$    \\
(3) & $\sqrt{K-\lambda_1}^{f_1}$    & $\sqrt{K^2-\lambda_2V}^{0}$ & $-\sqrt{K^2-\lambda_2V}^{g_2}$ & $-\sqrt{K-\lambda_1}^{f_2}$    \\
(4) & $\sqrt{K-\lambda_1}^{0}$    & $\sqrt{K^2-\lambda_2V}^{g_1}$ & $-\sqrt{K^2-\lambda_2V}^{g_2}$ & $-\sqrt{K-\lambda_1}^{f_2}$    \\
(5) & $\sqrt{K^2-\lambda_2V}^{g_1}$ & $\sqrt{K-\lambda_1}^{f_1}$    & $-\sqrt{K-\lambda_1}^{f_2}$    & $-\sqrt{K^2-\lambda_2V}^{0}$ \\
(6) & $\sqrt{K^2-\lambda_2V}^{g_1}$ & $\sqrt{K-\lambda_1}^{f_1}$    & $-\sqrt{K-\lambda_1}^{0}$    & $-\sqrt{K^2-\lambda_2V}^{g_2}$ \\
(7) & $\sqrt{K^2-\lambda_2V}^{g_1}$ & $\sqrt{K-\lambda_1}^{0}$    & $-\sqrt{K-\lambda_1}^{f_2}$    & $-\sqrt{K^2-\lambda_2V}^{g_2}$ \\
(8) & $\sqrt{K^2-\lambda_2V}^{0}$ & $\sqrt{K-\lambda_1}^{f_1}$    & $-\sqrt{K-\lambda_1}^{f_2}$    & $-\sqrt{K^2-\lambda_2V}^{g_2}$
\end{tabular}
\end{table}

\begin{lemma}\label{lm:fg}
The following equalities hold:
\[
f=\frac{1}{2}\left(v-1-\frac{2k+(v-1)(r+s)}{r-s}\right), \quad  
g=\frac{1}{2}\left(v-1+\frac{2k+(v-1)(r+s)}{r-s}\right).
\]
\end{lemma}
\begin{proof}
The multiplicities $f$ and $g$ of the eigenvalues $r$ and $s$, respectively, 
can be found from the equations $1 + f + g = v$ and $k + fr + gs = 0$ 
(see \cite[Section~1.1.4]{BH}). Thus, it suffices to verify that 
$f$ and $g$ given in the lemma satisfy these equations.
\end{proof}

It will be handy to use the following identity:
\[
f-g=
-\frac{2k+(v-1)(r+s)}{r-s}.
\]

\begin{lemma}\label{lm:prop3.2}
{\rm (\cite[Proposition 3.2]{HKM})}  
The following inequality holds:
\[
0\leq K + (g_1 - g_2)\sqrt{K^2-\lambda_2V}\leq m(n - 1).
\]
\end{lemma}

\begin{lemma}\label{lm:K}
    $K=k+s$.
\end{lemma}
\begin{proof}
As $C$ is a Hoffman coclique, every vertex from 
$\Gamma\setminus C$ is adjacent to precisely $-s$ vertices from $C$ 
(see, e.g., \cite[Proposition~1.1.7]{BW}). 
\end{proof}
\smallskip

\begin{lemma}\label{lm:case1}
    Case $(1)$ is not possible.
\end{lemma}
\begin{proof}
One has $r+s=\sqrt{K^2-\lambda_2V}$ and $s=-\sqrt{K^2-\lambda_2V}$, hence $r=-2s$. Further, 
$g_1=c-1$, $g_2=g-c$, hence 
$g_1+g_2=g-1=m-1$, and 
$f_2=0$, hence $f_1=m(n-1)$ by Eq. \eqref{eq:ddgmult}. Thus, $g<f$ or $n=2$, in which case $g=f$.
In either case, using $r=-2s$ and Lemma \ref{lm:fg} we obtain 
\begin{eqnarray*}
f-g=
\frac{2k+(v-1)(-s)}{3s}&\geq &0,
\end{eqnarray*}
a contradiction since $s<0$. 
\end{proof}

\begin{lemma}\label{lm:case4}
    Case $(4)$ is not possible.
\end{lemma}
\begin{proof}
    One has $r=\sqrt{K^2-\lambda_2V}$, $r+s=-\sqrt{K^2-\lambda_2V}$,  
    hence $s=-2r$. 
    Further, $g_1=f-c+1$, $g_2=c-1$, hence $g_1+g_2=f=m-1$, 
    and $f_1=0$, $f_2=g-c$, hence $f_1+f_2=g-c=m(n-1)$ 
    by Eq. \eqref{eq:ddgmult}. Thus, $g=c+m(n-1)>f$.

Using $s=-2r$ and Lemma \ref{lm:fg}, we obtain 
\begin{eqnarray*}
f-g=
\frac{2k+(v-1)(-r)}{-3r}&<&0,    
\end{eqnarray*}
whence $2k+(v-1)(-r)>0$, i.e., $r<\frac{2k}{v-1}<2$. 
Since $\Gamma$ has a Hoffman coclque, $s$ and hence $r$ are integral; thus, $r=1$ and $s=-2$. 
All strongly regular graphs with $s=-2$ are known 
(see \cite[Section~1.1.10]{BW} or \cite{S}); among them, 
only the $3\times 3$ grid $(9,4,1,2)$, the Petersen graph $(10,3,0,1)$ 
and its complement $(10,6,3,4)$ satisfy $r=1$. It is easily seen that 
these graphs do not admit a 
decomposition as in Theorem \ref{theo:main}.
\end{proof}

\begin{lemma}
    Case $(5)$ is not possible.
\end{lemma}
\begin{proof}
One has $r=\sqrt{K^2-\lambda_2V}$, $r+s=\sqrt{K-\lambda_1}$, $s=-\sqrt{K-\lambda_1}$, hence $r=-2s$. 
Further, $g_1=f-c+1$, $g_2=0$, $f_1=c-1$ and $f_2=g-c$. It follows from Eq. 
\eqref{eq:ddgmult} that 
$g=m(n-1)+1$ and $f=m+c-2$.
Using Lemma \ref{lm:fg}, we obtain 
$k=(3m+3c-5-v)s$.

Further, as $c=vs/(s-k)$ and $v=mn+c$, 
we find that $c$ satisfies the equation
$2c^2-c(mn-3m+5)+mn=0$, whence 
\begin{equation}\label{eq:c}
c=c_{1,2}=\frac{(mn-3m+5) \pm \sqrt{(mn-3m+5)^2-8mn}}{4}.    
\end{equation}

(1) Suppose first that $c=c_1$. Since 
$s^2=K-\lambda_1$, it follows by Lemma \ref{lm:K} that $k=s^2-s+\lambda_1\geq s^2-s$. 
Therefore, $c=vs/(s-k)\leq v/(2-s)$ and $s\geq 2-v/c=1-mn/c$. 
We claim that $s>-4$. 
Indeed, if $1-mn/c\leq -4$, then
\[
mn\geq 5c=
\frac{5}{4}\left((mn-3m+5) + \sqrt{(mn-3m+5)^2-8mn}\right),
\]
which does not hold for integral $m>1$ and $n>1$. 
Since $\Gamma$ contains a Hoffman coclique, $s$ must be an integer; 
thus, $s\in \{-2,-3\}$. 

Strongly regular graphs with $s=-2$ were determined by Seidel \cite{S}. 
In particular, if $s=-2$, $r=4$, then $\Gamma$ has one of the following 
parameter sets $(v,k,\lambda,\mu)$: 
\begin{itemize} 
    \item $(27,16,10,8)$ with $f=6$ and $g=20$. 
    Then $c=27(-2)/(-2-16)=3$. Further, it follows from $f = m + c-2$ 
    that $m=5$, which does not divide $mn=v-c=24$, a contradiction.
    \item $(28,12,6,4)$ with $f=7$ and $g=20$. 
    Then $c=28(-2)/(-2-14)=4$. It follows 
    that $m=5$, which does not divide $mn=v-c=24$.
    \item $(36,10,4,2)$. In this case $\Gamma$ is the $6\times 6$ grid, 
    which contains a Hoffman coclique $C$ of size $6$. 
    However, one can easily check that $\Gamma\setminus C$ is 
    not a divisible design graph.
\end{itemize}

Neumaier \cite[Page 399]{Neu} noted that there are exactly 64 parameter sets 
for strongly regular graphs with the least eigenvalue $s = -3$ and $\mu \notin \{k, 6, 9 \}$. 
They can be determined, e.g., by using the GAP package AGT \cite{Eva22}. 
Clearly, $\mu\ne k$, as $\Gamma$ is primitive, and if $k\in \{6, 9 \}$, the parameters of $\Gamma$ can be determined from 
the equations $k-\mu=-rs$ and
$\lambda=\mu+r+s$ (see \cite[Section~1.1.1]{BW}). 
In particular, if $s=-3$, $r=6$, then $\Gamma$ has one of the following 
parameter sets $(v,k,\lambda,\mu)$: 
\begin{itemize}
    \item $(70,27,12,9)$ with $f=20$ and $g=49$. 
    Then $c=70(-3)/(-3-27)=7$. It follows that $m=15$, 
    which does not divide $v-c=63$.
    \item $(81,60,45,42)$: this graph does not contain 
    a Hoffman coclique, as $81(-3)/(-3-60)$ is not integral.
    \item $(81,24,9,6)$ with $f=24$ and $g=56$. 
    Then $c=81(-3)/(-3-24)=9$. It follows that $m=17$, 
    which does not divide $v-c=72$.
\end{itemize}

(2) Suppose now that $c=c_2$ (see Eq. \eqref{eq:c}). 
Since $g_1-g_2=m-1$ and 
$r=\sqrt{K^2-\lambda_2V}=-2s$, 
it follows from Lemmas \ref{lm:prop3.2}, \ref{lm:K} 
that 
\[
k+s+(m-1)(-2s)\leq m(n-1).
\] 
Substituting $k=(3m+3c-5-v)s$ into the latter inequality, we obtain 
\[
s(-m(n-1)+2c-2)\leq m(n-1).
\]
Note that $-m(n-1)+2c-2<0$, 
as 
\[
c=\frac{(mn-3m+5) - \sqrt{(mn-3m+5)^2-8mn}}{4}<\frac{m(n-3)+5}{2}.
\]
Therefore, 
\[
s\geq -\frac{m(n-1)}{m(n-1)-2(c-1)}.
\]
We claim that $s>-2$. 
Indeed, if, on the contrary, 
$\frac{m(n-1)}{m(n-1)-2(c-1)}\geq 2$, then:
\begin{align*}
m(n-1)\geq 2(m(n-1)-2(c-1)),\\
m(n-1)\geq 2m(n-1)-\left((mn-3m+1) - \sqrt{(mn-3m+5)^2-8mn}\right),\\
m(n-1)\geq m(n+1)-1 + \sqrt{(mn-3m+5)^2-8mn},
\end{align*}
which does not hold. 
Again, as $s$ must be an integer, it follows that $s\geq -1$, i.e., 
$\Gamma$ is a complete graph, a contradiction. 
The lemma is proved. 
\end{proof}

\begin{lemma}
    Case $(8)$ is not possible.
\end{lemma}
\begin{proof}
    One has $r=\sqrt{K-\lambda_1}$, $r+s=-\sqrt{K-\lambda_1}$, $s=-\sqrt{K^2-\lambda_2V}$, hence $s=-2r$. Further, 
    $f_1=f-c+1$, $f_2=c-1$, $g_1=0$, $g_2=g-c$. 
    It follows from Eq. \eqref{eq:ddgmult} that 
    $f=m(n-1)$ and $g=m+c-1$.  
    Using Lemma \ref{lm:fg}, we obtain 
    $k=(-s/2)(-mn+3m+2c-2)$.

    Further, from 
    \[
    c=\frac{vs}{s-k}=\frac{(mn+c)s}{(s-(-s/2)(-mn+3m+2c-2))},
    \]
    we find that $c$ satisfies the equation
    $2c^2-c(mn-3m+2)-2mn=0$, whence (as $c>0$)
    \[
    c=\frac{(mn-3m+2) + \sqrt{(mn-3m+2)^2+16mn}}{4}.
    \]
    Since $g_1-g_2=-(g-c)$ and $s=-\sqrt{K^2-\lambda_2V}$, it follows 
    from Lemmas \ref{lm:prop3.2}, \ref{lm:K} that
    \begin{equation}\label{eq:case81}
    0\leq k+s+(g-c)s=k+ms,    
    \end{equation}
    whence $m\leq -k/s$ (recall that $s<0$). 
    Substituting $k=(-s/2)(-mn+3m+2c-2)$ into the latter inequality, we obtain
    \begin{align}
    mn\leq m+2c-2,\label{eq:case82}\\
    mn \leq m + \frac{(mn-3m+2)+\sqrt{(mn-3m+2)^2+16mn}}{2} - 2,\nonumber\\
    mn+m+2 \leq \sqrt{(mn-3m+2)^2+16mn},\nonumber
    \end{align}
    hence $(n-1)(m-2) \leq 0$. Since $m>1$ and $n>1$, 
    the only solution is $m=2$ and then $c=n$. 
    Moreover, since Eq. \eqref{eq:case82} holds with equality, 
    so does Eq. \eqref{eq:case81}; thus, $k=-2s=4r$. 
    Recall that the parameter $\mu$ of $\Gamma$ satisfies 
    $\mu=k+rs$ (see \cite[Section~1.1.1]{BW}); 
    hence $\mu=4r+r(-2r)=4r-2r^2$ and $r\leq 2$. 
    If $r=1$, then $\Gamma$ is the $3\times 3$ grid $(9,4,1,2)$, 
    a contradiction as in Lemma \ref{lm:case4}. 
    If $r=2$, then $k\leq 8$; there are no feasible parameters 
    sets of strongly regular graphs, satisfying these conditions, 
    a contradiction.
\end{proof}
\smallskip

In order to proceed with the remaining cases when $r+s=0$, let us recall 
that a {\bf Deza graph} \cite{EFHHH98,DezaSurvey,DezaPage} with parameters $(v,k,b,a)$ 
is a graph on $v$ vertices, which is $k$-regular and every two different vertices have 
either $b$ or $a$ common neighbours (we always assume that $b\geq a$). 
For example, a divisible design graph is a Deza graph. 
Also recall that the {\bf composition} $G_1[G_2]$
of graphs $G_1$ and $G_2$ (with vertex sets $V_1,V_2$, respectively) 
is a graph on $V_1\times V_2$, in which the adjacency relation $\sim$ is 
defined by $(x_1,x_2)\sim (y_1,y_2)$ if and only if $x_1\sim_{G_1} y_1$ or 
$(x_1=y_1$ and $x_2\sim_{G_2} y_2)$.

\begin{prop}\label{k=b}
{\rm (\cite[Theorem 2.6]{EFHHH98})} 
Let $D$ be an $(v,k,b,a)$ Deza graph. Then $b=k$ if and only if $D$ is 
the composition $G_1[G_2]$, where $G_1$ is a strongly regular graph 
with parameters $(v_1,k_1,\lambda,\mu)$ with $\lambda=\mu$, and $G_2$ 
is $\overline{K_{v_2}}$ for some integers $v_1$, $k_1$, $\lambda$, $v_2$. 
Moreover, the parameters of $D$ satisfy $v=v_1v_2$, $k=b=k_1v_2$, 
$a=\lambda v_2$ and $v_2=\cfrac{k^2-a v}{k-a}$.
\end{prop}

Note that the composition $G_1[G_2]$ from Proposition \ref{k=b}
is a divisible design graph with parameters 
$(v_1v_2,k_1v_2,k_1v_2,\lambda v_2; v_1,v_2)$.

\begin{lemma}\label{lm:case67}
    $r,s\ne \pm \sqrt{K^2 -\lambda_2 V}$, in particular, 
    Cases $(6)$ and $(7)$ are not possible.
\end{lemma}
\begin{proof}
    Since $r=-s=\sqrt{K^2 -\lambda_2 V}$, one has  
    $r+s=\sqrt{K-\lambda_1}=0$. 
Since $K=\lambda_1$, $\Delta$ is a Deza graph with $K=\lambda_1$ 
(where $\lambda_1$ plays the role of the parameter $b$). 
By Proposition \ref{k=b}, $\Delta$ is the composition $G_1[G_2]$ 
of a strongly regular graph $G_1$ and a coclique $G_2$. 
More precisely, since two vertices in $\Delta$ have $K$ common neighbours 
whenever they belong to the same copy of $G_2$, $G_1$ has 
parameters $(m,K/n,\lambda_2/n,\lambda_2/n)$ and $G_2$ is $\overline{K_{n}}$. 

Furthermore, $f_1 + f_2 = m(n-1) = c-1$ holds and hence $c = m(n-1)+1$. 
Since $c=(mn+c)s/(s-k)$, it follows that 
$s=-k(mn-m+1)/mn$.  
As $s=-\sqrt{K^2 -\lambda_2 V}$, we obtain 
$s^2=(k+s)^2-\lambda_2 V$ and $k(k+2s)=\lambda_2 V>0$, 
hence $k\geq -2s$.
Therefore, $k\geq 2k(mn-m+1)/mn$, whence 
$2(m-1)\geq mn$, contrary to $n>1$.
\end{proof}

\begin{lemma}\label{lm:maincase}
    Suppose that $r,s=\pm \sqrt{K -\lambda_1}$. Then 
    $c=m$ and the parameters of $\Gamma$ satisfy:
\[
v=\frac{(-s)(n^2-1)}{n+s},\quad 
k=(-s)n,\quad 
\lambda=\mu=(-s)(n+s).
\]
\end{lemma}
\begin{proof}
Since $r+s=0$, it follows that $\pm\sqrt{K^2-\lambda_2V}=0$ and hence $g_1+g_2=m-1$. 
It then follows from Eq. \eqref{eq:ddgmult} that $c=m$. 
From $v=mn+m=m(n+1)$ and $c=vs/(s-k)$, we obtain $k=(-s)n$. 
Using the equations $k-\mu=-rs$ and $\lambda=\mu+r+s$ (see \cite[Section~1.1.1]{BW}), 
we further find that $\lambda=\mu=(-s)(n+s)$.
\end{proof}
\smallskip

Theorem \ref{theo:main} now follows from Lemmas \ref{lm:case1}--\ref{lm:maincase}.
Next, we will prove Corollary \ref{coro}.

\begin{lemma}\label{lm:Delta}
    In the situation of Lemma \ref{lm:maincase}, the parameters of $\Delta$ satisfy 
    $m=(-s)(n-1)/(n+s)$ and: 
\begin{itemize}
    \item $V=n(-s)(n-1)/(n+s)$,
    \item $K=(-s)(n-1)$,
    \item $\lambda_1=(-s)(n+s-1)$,
    \item $\lambda_2=(-s)(n-1)(n+s)/n$.
\end{itemize}
\end{lemma}
\begin{proof}
Using the equality $k(k-\lambda-1)=\mu(v-k-1)$ (see \cite[Section~1.1.1]{BW}), 
we obtain $m=(-s)(n-1)/(n+s)$. Then $V=mn=n(-s)(n-1)/(n+s)$ and $K=k+s=(-s)(n-1)$. 
Further, $K-\lambda_1=s^2$ yields $\lambda_1=(-s)(n+s-1)$. 
As $K^2-\lambda_2V=0$, we obtain $\lambda_2=\frac{s^2(n-1)^2}{mn}=(-s)(n-1)(n+s)/n$. 
\end{proof}

\begin{lemma}\label{lm:sprime}
    In the situation of Lemma \ref{lm:maincase}, suppose that 
    $(-s)$ is a prime power. Then $s=-q^d$ and $n=q^{d+1}$ 
    for some natural number $d$, where $q=n/(-s)$.
\end{lemma}
\begin{proof}
    As $\lambda_2$ is integral, it follows from Lemma \ref{lm:Delta} that 
    $n$ divides $s^2$, and thus $n$ is a prime power. 
    Put $q=n/(-s)$. 
    Then $m=(-s)(n-1)/(n+s)=\frac{n-1}{\left(\frac{n}{-s}\right)-1}$ 
    implies that $n$ is an integer power of $q$, say, $n=q^{d+1}$, 
    which completes the proof of the lemma and Corollary \ref{coro}.
\end{proof}

\section{A construction of $\Gamma$ and its discussion}\label{sect:4}

\subsection{The quotient matrix of the canonical partition of $\Delta$}

Recall that the canonical partition of a divisible design graph (with $m$ classes) is equitable (see~\cite[Theorem 2.2]{CH}). This means that its adjacency matrix $A$ decomposes into blocks $A_{ij}$, $i,j\in [m]$, where $[m]:=\{1, 2,\dots ,m\}$, according to the canonical partition, and each block 
$A_{ij}$ has constant row and column sums. Let $R = (r_{ij})$ be the quotient matrix of $A$, 
where $r_{ij}$ is the row sum of $A_{ij}$. By~\cite[Proposition 3.2]{HKM}, the matrix $R$ 
is symmetric and its entries satisfy: 
\begin{itemize}
\item $\displaystyle{\sum_{i\in [m]}r_{ij} =K\quad \mathrm{for\, all}\,\, j\in [m]}$,
\item $\displaystyle{\sum_{i,j\in [m]}r_{ij}^2 =\mathrm{trace}(R^2)=mK^2 -(m-1)\lambda_2 V,}$
\item $0\leq \mathrm{trace}(R)=K+(g_1-g_2)\sqrt{K^2 -\lambda_2 V}\leq m(n-1).$
\end{itemize}

Note that the eigenvalues of $R$ are $K$, $\pm\sqrt{K^2 -\lambda_2 V}$ with multiplicities $1$, $g_1$ and $g_2$, respectively. 

\begin{lemma}\label{quotmat}(cf. \cite[Theorem 3.3]{HKM})
Let $\Delta$ be a divisible design graph with parameters as in Lemma~\ref{lm:Delta}. 
Then $R=(n+s)J$, where $J$ is the all-ones $m\times m$-matrix.
\end{lemma}
\begin{proof}
Since $\sqrt{K^2 -\lambda_2 V}=0$, we have $\mathrm{trace}(R^2) = K^2 = (-s)^2(n-1)^2$ and 
the following equations:
\begin{itemize}
\item $\displaystyle{\sum_{i\in [m]}r_{ij} = (-s)(n-1) \quad \mathrm{for\, all}\,\, j\in [m]}$,
\item $\displaystyle{\sum_{i,j\in [m]}r_{ij}^2  = (-s)^2 (n-1)^2}$,
\item $\displaystyle{\sum_{i\in [m]}r_{ii}  = (-s)(n-1).}$
\end{itemize}
Then 
\begin{align*}
    \sum_{i,j}\left(r_{i,j}-\frac{(-s)(n-1)}{m}\right)^2=
    \sum_{i,j}r_{i,j}^2-2\frac{(-s)(n-1)}{m}\sum_{i,j}r_{i,j}+(-s)^2(n-1)^2=\\
    (-s)^2(n-1)^2-2\frac{(-s)(n-1)}{m}m(-s)(n-1)+(-s)^2(n-1)^2=0.
\end{align*}

Thus, $r_{ij} =(-s)(n-1)/m = (n+s)$ for all $i,j \in [m]$.
\end{proof}

\subsection{A construction of $\Gamma$}
We shall show that a strongly regular graph 
satisfying the hypothesis of Theorem \ref{theo:main}
comes from the following Construction \ref{SRG} below.
Recall that the parameters of a symmetric 2-design $(\mathcal{P},\mathcal{B})$ 
with point set $\mathcal{P}$ and block set $\mathcal{B}$ are written as the tuple
$(|\mathcal{P}|,|B|,|B\cap B'|)$, which does not depend on the particular choice of 
blocks $B,B'\in \mathcal{B}$.

\begin{construction}\label{SRG}
Let $n$ be a natural number and $s$ be a negative integer such that there exists a divisible design graph $\Delta$ with $m=(-s)(n-1)/(n+s)$ canonical classes, say $P_1,\dots, P_m$, each of size $n$, 
and with parameters $V, K, \lambda_1, \lambda_2$ as in Lemma \ref{lm:Delta}.

Let $\mathcal{D}=(\mathcal{P},\mathcal{B})$ be a symmetric $2$-design with parameters $(m, (-s), (-s)(n+s)/n)$ and $\phi$ be an arbitrary bijection from the set of all canonical classes 
of $\Delta$ to the set $\mathcal{B}$ of blocks.

Let $\Gamma$ be a graph defined as follows:
\begin{itemize}
   \item The vertex set of $\Gamma$ is the union of the vertex set of $\Delta$ and $\mathcal{P}$. 
   \item Two different vertices from $\Delta$ are adjacent in $\Gamma$ if and only if 
   they are adjacent in $\Delta$. The subgraph of $\Gamma$ induced by $\mathcal{P}$ is an $m$-coclique. 
   A vertex from a canonical class $P_i$ of $\Delta$ is adjacent to a vertex $y\in \mathcal{P}$ 
   if and only if $y$ belongs to the block $\phi(P_i)$ in $\mathcal{D}$. 
   \end{itemize}
\end{construction}

Recall that $\Gamma(x)$ stands for the set of neighbours of a vertex $x$ of a graph~$\Gamma$.

\begin{theorem}\label{Th2}
A graph $\Gamma$ from Construction~\ref{SRG} is a strongly regular graph with parameters 
$v=m(n+1)$, $k=(-s)n$, $\lambda=\mu=(-s)(n+s)$. 
\end{theorem}
\begin{proof}
The number of vertices of $\Gamma$ is $mn+m$.
If $x\in P_i$, then there are $K + |\phi(P_i)| = (-s)n$ vertices in $\Gamma(x)$. 
If $y\in \mathcal{P}$, then by the construction 
$\Gamma(y)$ consists of the vertices of $(-s)$ canonical classes of $\Delta$, of size $n$ each. 
Hence, $\Gamma$ is a regular graph of degree~$(-s)n$.

Let $x$ and $y$ be two different vertices from $\Gamma$. Then:
\begin{itemize}
    \item[(1)] if $x,y$ belong to the same canonical class of $\Delta$,
    then they have $\lambda_1 = (-s)(n+s-1)$ common neighbours in $\Delta$ and $(-s)$ 
    common neighbours in~$\mathcal{P}$. Thus, $|\Gamma(x)\cap\Gamma(y)|=(-s)(n+s)$.
    \item[(2)] if $x,y$ belong to different canonical classes of $\Delta$, then they 
    have $\lambda_2=(-s)(n-1)(n+s)/n$ common neighbours in $\Delta$ and $(-s)(n+s)/n$ 
    common neighbours in~$\mathcal{P}$. Thus, $|\Gamma(x)\cap\Gamma(y)|=(-s)(n+s)$.
    \item[(3)] if $x,y$ are in $\mathcal{P}$, then their neighborhoods share $(-s)(n+s)/n$
    canonical classes of $\Delta$, with $n$ vertices in each of them. 
    Thus, $|\Gamma(x)\cap\Gamma(y)|=(-s)(n+s)$.
    \item[(4)] finally, suppose that $x\in \Delta$ and $y\in \mathcal{P}$. 
    By Lemma~\ref{quotmat}, the quotient matrix of the canonical partition of $\Delta$ equals $(n+s)J$.
    Since $\Gamma(y)$ consists of $-s$ canonical classes of $\Delta$, 
    it follows that $x$ and $y$ have $(n+s)$ times $(-s)$ common neighbours in $\Delta$ 
    (and no common neighbours in~$\mathcal{P})$. 
\end{itemize}

Therefore, in all cases the number of common neighbours of two different vertices 
in $\Gamma$ equals $(-s)(n+s)$. This completes the proof of Theorem~\ref{Th2}.
\end{proof} 

\begin{lemma}\label{cn} 
Let $\Gamma$ be a graph satisfying the hypothesis of Theorem \ref{theo:main} and 
$\Delta$ a divisible design graph induced on $\Gamma\setminus C$. 
\begin{itemize}
\item[$(1)$] If $x$ and $y$ are vertices from the same canonical class of $\Delta$, 
then $\Gamma(x)\cap C=\Gamma(y)\cap C$. 
\item[$(2)$] If $z$ is a vertex in $C$, then every canonical class of $\Delta$ 
is either contained in or disjoint with $\Gamma(z)$.
\end{itemize}
\end{lemma}
\begin{proof} 
Let $x$ and $y$ be any two vertices from the same canonical class of $\Delta$.
By Lemmas~\ref{lm:maincase} and~\ref{lm:Delta}, $\mu - \lambda_1 = (-s)$. Thus, $x$ and $y$ 
have exactly $(-s)$ common neighbours in $C$, which are all their neighbours in $C$ as 
$k-K=(-s)$ by Lemma \ref{lm:K}. This shows $(1)$, and $(2)$ follows immediately from $(1)$.
\end{proof}

\begin{theorem}\label{Th3}
Let $\Gamma$ be a graph satisfying the hypothesis of Theorem \ref{theo:main}. 
Then $\Gamma$ can be obtained from Construction \ref{SRG}.
\end{theorem}
\begin{proof}
    It suffices to verify that there exist a design $\mathcal{D}$ and a bijection $\phi$: 
    their existence follows from Lemma \ref{cn} (observe that two different vertices 
    from $C$ share $\mu/n=(-s)(n+s)/n$ canonical classes of $\Delta$).
\end{proof}

\subsection{Concluding remarks}

Divisible design graphs with the parameters required in Construction~\ref{SRG} were first 
found in~\cite{VK} 
for $n=q^d$, $s=-q^{d-1}$, where $q$ is any prime power and $d$ is any positive integer $>1$, 
using affine designs and some ideas of Wallis~\cite{WW}, Fon-Der-Flaass~\cite{FF}, 
and Muzychuk~\cite{MM} constructions of strongly regular graphs. 
Construction~\ref{SRG} generalizes the construction of strongly regular graphs 
which was given in~\cite{VKa}. 
Note that there are exactly 28 strongly regular graphs with parameters $(40,27,18,18)$ by \cite{ES}: 
27 of them satisfy the hypothesis of Theorem \ref{theo:main} with a divisible design graph 
having parameters $(36,24,15,16;4,9)$. Removing different Hoffman cocliques from the same strongly 
regular graph may give non-isomorphic divisible design graphs; in particular, in this way 
one obtains 87 divisible design graphs with parameters $(36,24,15,16;4,9)$, only four of which 
come from the construction in \cite{VK} using affine planes. 

All known examples of strongly regular graphs with the parameters as in Theorem~\ref{Th2} 
and divisible design graphs with the parameters as in Construction~\ref{SRG} have 
the least eigenvalue $s$, which is a prime power.

Let us discuss the first case when $s$ is not a prime power, i.e., $s=-6$.
Then $\frac{6(n-1)}{n-6}$ and $\frac{6(n-6)}{n}$ must be integers 
(see the parameters $m$ and $\lambda_2$ in Lemma \ref{lm:Delta}). Thus, $n\in \{9,12,36\}$.
If $n=9$, then $n+s=3$; hence, by Lemma \ref{quotmat}, a subgraph of $\Delta$ induced by 
a canonical class would have been a cubic graph on 9 vertices, which is impossible.

If $n=12$, then the parameters of $\Delta$ and $\Gamma$ would be $(132, 66, 30, 33; 11, 12)$ and 
$(143, 72, 36, 36)$, respectively. Strongly regular graphs with such parameters do exist \cite{BW,VeK}. 
However, as far as we could check, the examples known to us do not contain Hoffman cocliques; 
thus, they do not come from Construction \ref{SRG}.

Finally, if $n=36$, then $m=7$ and the parameters of $\Delta$ and $\Gamma$ would be 
$(252, 210, 174, 175; 7, 36)$ and $(259, 216, 180, 180)$, respectively.
The existence of both graphs is unknown to us 
(and they cannot be constructed by using affine designs, 
as there is no affine plane of order $6$ \cite{T1,T2}).
Construction \ref{SRG} shows that the existence of the former 
one implies the existence of the latter one.

\bibliographystyle{amsplain}
\bibliography{refer}

\providecommand{\bysame}{\leavevmode\hbox to3em{\hrulefill}\thinspace}
\providecommand{\MR}{\relax\ifhmode\unskip\space\fi MR }
\providecommand{\MRhref}[2]{%
  \href{http://www.ams.org/mathscinet-getitem?mr=#1}{#2}
}
\providecommand{\href}[2]{#2}
\begin{thebibliography}{10}

\bibitem{AH}
A.~Abiad and W.~H. Haemers, \emph{Switched symplectic graphs and their
  2-ranks}, Des. Codes Cryptogr. \textbf{81} (2016), 35--41.

\bibitem{BH}
A.~E. Brouwer and W.~H. Haemers, \emph{Spectra of graphs}, Springer, New York,
  2012.

\bibitem{BIK}
A.~E. Brouwer, F.~Ihringer, and W.~M. Kantor, \emph{Strongly regular graphs
  satisfying the 4-vertex condition}, arXiv (2021),
  \url{https://arxiv.org/abs/2107.00076}.

\bibitem{BW}
A.~E. Brouwer and H.~Van Maldeghem, \emph{Strongly regular graphs},
  Encyclopedia of Mathematics and its Applications, Cambridge University Press,
  Cambridge, 2022.

\bibitem{CH}
D.~Crnkovi\'c and W.~H. Haemers, \emph{Walk-regular divisible design graphs},
  Des. Codes Cryptogr. \textbf{72} (2014), 165--175.

\bibitem{D}
P.~Delsarte, \emph{An algebraic approach to the association schemes of coding
  theory}, Philips Research Reports Suppl. \textbf{10} (1973).

\bibitem{EFHHH98}
M.~Erickson, S.~Fernando, W.~H. Haemers, D.~Hardy, and J.~Hemmeter, \emph{Deza
  graphs: A generalization of strongly regular graphs}, J. Comb. Designs
  \textbf{7} (1999), 359--405.

\bibitem{Eva22}
R.~J. Evans, \emph{{AGT, Algebraic Graph Theory}}, 2022,
  \url{https://gap-packages.github.io/agt}.

\bibitem{FF}
D.~G. Fon-Der-Flaass, \emph{New prolific constructions of strongly regular
  graphs}, Adv. Geom. \textbf{2} (2002), 301--306.

\bibitem{DezaSurvey}
S.~Goryainov and L.~V. Shalaginov, \emph{Deza graphs: a survey and new
  results}, arXiv (2021), \url{https://arxiv.org/abs/2103.00228}.

\bibitem{H}
W.~H. Haemers, \emph{Eigenvalue techniques in design and graph theory}, Ph.D.
  thesis, Technische Universiteit Eindhoven, 1979.

\bibitem{HH}
W.~H. Haemers and D.~G. Higman, \emph{Strongly regular graphs with strongly
  regular decomposition}, Linear Algebra Appl. \textbf{114--115} (1989),
  379--398.

\bibitem{HKM}
W.~H. Haemers, H.~Kharaghani, and M.~Meulenberg, \emph{Divisible design
  graphs}, J. Comb. Theory Ser. A \textbf{118} (2011), 978--992.

\bibitem{AJH}
A.~J. Hoffman, \emph{On eigenvalues and colorings of graphs}, Graph Theory and
  its Applications (B.~Harris, ed.), Academic Press, New York, 1970,
  pp.~79--91.

\bibitem{Ih}
F.~Ihringer, \emph{A switching for all strongly regular collinearity graphs
  from polar spaces}, J. Algebr. Comb. \textbf{46} (2017), 263--274.

\bibitem{VK}
V.~V. Kabanov, \emph{New versions of the {W}allis--{F}on-{D}er-{F}laass
  construction to create divisible design graphs}, Discrete Math. \textbf{345}
  (2022), 113054.

\bibitem{VKa}
\bysame, \emph{A new construction of strongly regular graphs with parameters of
  the complement symplectic graph}, Electron. J. Comb. \textbf{30(1)} (2023),
  no.~P1.25.

\bibitem{VeK}
V.~Krcadinac, \emph{Steiner 2-designs},
  \url{https://web.math.pmf.unizg.hr/\%7Ekrcko/results/steiner.html}.

\bibitem{Ku}
S.~Kubota, \emph{Strongly regular graphs with the same parameters as the
  symplectic graph}, Siberian Electronic Mathematical Reports \textbf{13}
  (2016), 1314--1338.

\bibitem{MM}
M.~Muzychuk, \emph{A generalization of {W}allis--{F}on-{D}er-{F}laass
  construction of strongly regular graphs}, J. Algebr. Comb. \textbf{25}
  (2007), 169--187.

\bibitem{Neu}
A.~Neumaier, \emph{Strongly regular graphs with smallest eigenvalue $-m$},
  Arch. Math.(Basel) \textbf{33(4)} (1979--80), 392--400.

\bibitem{DezaPage}
Dmitry Panasenko, \emph{{Strictly Deza Graphs}}, 2023,
  \url{http://alg.imm.uran.ru/dezagraphs/main.html}.

\bibitem{S}
J.~J. Seidel, \emph{Strongly regular graphs with $(-1, 1, 0)$ adjacency matrix
  having eigenvalue 3}, Linear Algebra Appl. \textbf{1} (1968), 281--298.

\bibitem{ES}
E.~Spence, \emph{The strongly regular (40,12,2,4) graphs}, Electron. J. Comb.
  \textbf{7(1)} (2000), no.~R22.

\bibitem{T1}
G.~Tarry, \emph{Le problème de 36 officiers}, Comptes Rendus de l'Association
  Française pour l'Avancement des Sciences \textbf{1} (1900), 122--123.

\bibitem{T2}
\bysame, \emph{Le problème de 36 officiers}, Comptes Rendus de l'Association
  Française pour l'Avancement des Sciences \textbf{2} (1901), 170--203.

\bibitem{vD}
E.~R. {van Dam}, \emph{Regular graphs with four eigenvalues}, Linear Algebra
  Appl. \textbf{226--228} (1995), 139--162.

\bibitem{WW}
W.~D. Wallis, \emph{Construction of strongly regular graphs using affine
  designs}, Bulletin of the Australian Mathematical Society \textbf{4} (1971),
  41--49.

\end{thebibliography}

\end{document}